\documentclass[11pt,a4paper,intlimits]{amsart}
\usepackage{amssymb}
\usepackage{amsfonts}
\usepackage{amsxtra}
\usepackage{color}

\usepackage{enumitem}

\usepackage[hyperref,frak]{paper_diening}

\newtheorem{theorem}{Theorem}[section]
\newtheorem{lemma}[theorem]{Lemma}
\newtheorem{proposition}[theorem]{Proposition}
\newtheorem{corollary}[theorem]{Corollary}

\newtheorem{definition}[theorem]{Definition}
\newtheorem{assumption}[theorem]{Assumption}

\newtheorem{remark}[theorem]{Remark}

\newcommand{\Rn}{\setR^n}
\newcommand{\RN}{\setR^N}
\newcommand{\RNn}{\setR^{N\times n}}
\newcommand{\Hom}{{\rm Hom}}
\newcommand{\kappaA}{{\kappa_{\!\mathcal{A}}}}
\newcommand{\TA}{T_{\mathcal{A}}}
\newcommand{\TAbar}{T_{\overline{\mathcal{A}}}}

\newcommand{\meantmp}[2]{#1\langle{#2}#1\rangle}
\newcommand{\mean}[1]{\meantmp{}{#1}}

\newcommand{\medint}{\dashint}

\def\calA{\mathcal{A}}

\def\la1{\lambda_1}

\begin{document}


\author{Lars Diening$^1$}
\address{$^1$ LMU Munich, Institute of Mathematics,
    Theresienstr.~39, 80333~Munich, Germany. email:
    lars.diening@mathematik.uni-muenchen.de. The work was partially
    supported by Gnampa.} 
\author{Daniel Lengeler$^2$}
\address{$^2$ Universit{\"a}t Regensburg, Fakult{\"a}t f{\"u}r Mathematik, Universit{\"a}tsstr. 31, 93053 Regensburg,
    Germany. email: daniel.lengeler@mathematik.uni-regensburg.de. The
    work was partially supported by the European Research Council
    under FP7, Advanced Grant n. 226234 ''Analytic Techniques for
    Geometric and Functional Inequalities'' .}
\author{Bianca Stroffolini$^3$}
\address{$^3$ Dipartimento di Matematica, Universit{\`a} di
    Napoli, Federico II, Via Cintia, 80126 Naples, Italy. email:
    bstroffo@unina.it. The work was partially
    supported by PRIN Project: ``Calcolo delle variazioni e Teoria
    Geometrica della Misura''.} 
\author{Anna Verde$^4$}
\address{$^4$ Dipartimento di
    Matematica, Universit{\`a} di Napoli, Federico II, Via Cintia, 80126
    Naples, Italy. email: anverde@unina.it}

\title{Partial regularity for minimizers of quasiconvex functionals
  with general growth}

\begin{abstract}
  We prove a partial regularity result for local minimizers of
  quasiconvex variational integrals with general growth. The main tool
  is an improved $\mathcal{A}$-harmonic approximation, which should be
  interesting also for classical growth.
\end{abstract}



\maketitle

\pagestyle{myheadings}
\thispagestyle{plain}
\markboth{L.~Diening, D.~Lengeler, B.~Stroffolini, A.~Verde}{Partial regularity for quasiconvex functionals}

\section{Introduction}
\label{sec:introduction}

In this paper we study  partial regularity for vector-valued   minimizers $u:\Omega\to \mathbb{R}^N$  of variational
 integrals:

\begin{align}\label{funct}
  \mathcal{F}(u) := \int_\Omega f(\nabla u)\,dx,
\end{align}
where $\Omega \subset \Rn$ is a domain and $f\,:\, \RNn \to \setR$
is a continuous function.

Let us recall Morrey's notion of quasiconvexity \cite{Mor52}:
\begin{definition}
$f$ is called quasiconvex if and only if 
\begin{align}
\medint_{B_1}f( A +\nabla \bfxi)\,dx \geq f(A)
\end{align}
holds for every $A \in \mathbb{R}^{nN}$ and every smooth $\bfxi: B_1
\to \mathbb{R}^N$ with compact support in the open unit ball $B_1$ in
$\mathbb{R}^n$.
\end{definition} 
By Jensen's inequality, quasiconvexity is a generalization of
convexity. It was originally introduced as a notion for proving the
lower semicontinuity and the existence of minimizers of variational
integrals. In fact, assuming a power growth condition, quasiconvexity
is proved to be a necessary and sufficient condition for the
sequential weak lower semicontinuity on $W^{1,p}(\Omega.
\mathbb{R}^N)$, $p>1$, see \cite{Mar85} and \cite{AceF84}. For general
growth condition see \cite{Foc97} and \cite{VerZ09}.  In the
regularity issue, a stronger definition comes into play.  In the
fundamental paper~\cite{Eva86} Evans considered strictly quasi-convex
integrands $f$ in the quadratic case and proved that if $f$ is of
class $C^2$ and has bounded second derivatives then any minimizing
function $\bfu$ is of class $C^{1,\alpha}(\Omega \setminus \Sigma)$
where $\Sigma$ has $ n$-dimensional Lebesgue measure zero. In
\cite{AceF84}, this result was generalized to integrands $f$ of
$p$-growth with $p\geq 2$ while the subquadratic growth was considered
in \cite{Carfusmin98}.  

In order to treat the general growth case, we introduce the notion of
{\em strictly $W^{1,\phi}$- quasiconvex} function, where $\phi$ is a
suitable N-function, see Assumption~\ref{ass:phi}.
\begin{definition}
  The function $f$ is {\em strictly $W^{1,\phi}$- quasiconvex} if and
  only if
  \begin{align*}
    \int_B f(\bfQ + \nabla \bfw) - f(\bfQ) \,dx &\geq k \int_B
    \phi_{\abs{\bfQ}}(\abs{\nabla \bfw})\,dx,
  \end{align*}
  for all balls $B \subset \Omega$, all $\bfQ \in \RNn$ and all $\bfw
  \in C^1_0(B)$, where $\phi_a(t) \sim \phi''(a+t)\,t^2$ for $a,t
  \geq 0$. A precise definition of $\phi_a$ is given in
  Section~\ref{sec:notation}.
\end{definition}
We will work with the following set of assumptions:
\begin{enumerate}[label={\bf (H\arabic{*})}]
\item \label{itm:freg} $f\in C^1(\Rn)\cap C^2(\Rn\setminus\{0\})$,
\item \label{itm:fbnd}for all $\bfQ \in \RNn$  it  holds
  \begin{align*}
    \abs{f(\bfQ)} \leq K \phi(\abs{\bfQ}),
  \end{align*}
\item \label{itm:fquasi} the function $f$ is {\em strictly
    $W^{1,\phi}$-quasiconvex};
\item \label{itm:D2fbnd}for all $\bfQ \in \RNn \setminus \set{\bfzero}$ 
  \begin{align*}
    \bigabs{(D^2f)(\bfQ)} &\leq c\, \phi''(\abs{\bfQ})
  \end{align*}
\item  \label{itm:D2fdiff} the following
  H{\"o}lder continuity of $D^2 f$ away from ~$\bfzero$
  \begin{align*}
     \bigabs{D^2f(\bfQ) - D^2f(\bfQ+\bfP)} &\leq c\, \phi''(\abs{\bfQ})
    \abs{\bfQ}^{-\beta} \abs{\bfP}^{\beta}
  \end{align*}
  holds for all $\bfP, \bfQ \in \RNn$ such that $ \abs{\bfP} \leq \frac{1}{2} \abs{\bfQ}$.
\end{enumerate}
Due to~\ref{itm:fbnd}, $\mathcal{F}$ is well defined on the Sobolev-Orlicz space $W^{1,\phi}(\Omega, \mathbb{R}^N)$, see section 2. Let us observe that assumption \ref{itm:D2fdiff} has been used to
show
everywhere regularity of radial functionals with $\phi$-growth,
\cite{DieSV09}. Following the argument given in \cite{Giu94} it is possible to prove
that \ref{itm:fquasi} implies the following {\em
  strong Legendre-Hadamard condition}
\begin{align*}
  (D^2f)(\bfQ)(\bfeta \otimes \bfxi, \bfeta \otimes \bfxi) &\geq c\,
  \phi''(\abs{\bfQ}) \abs{\bfeta}^2 \abs{\bfxi}^2
\end{align*}
for all $\bfeta \in \RN$, $\bfxi \in \Rn$ and $\bfQ \in \RNn \setminus
\set{\bfzero}$. Furthermore, \ref{itm:fquasi} implies that the functional
\begin{align*}
  \mathcal{J}(t) := \int_B f(\bfQ + t \nabla \bfw) - f(\bfQ) -
  k \phi_{\abs{\bfQ}}(t\abs{\nabla \bfw})\,dx
\end{align*}
attains its minimal value at $t=0$. Hence $\mathcal{J}''(0) \geq 0$,
that is
\begin{align}
  \label{eq:fmon}
  \begin{aligned}
    \int_B (D^2 f)(\bfQ) (\nabla \bfw, \nabla \bfw)\,dx &\geq k \int_B
    \phi_{\abs{\bfQ}}''(0) \abs{\nabla \bfw}^2\,dx
    \geq c\, \phi''(\abs{\bfQ}) \int_B \abs{\nabla \bfw}^2\,dx.
  \end{aligned}
\end{align}

As usual, the strategy for proving partial regularity consists
in showing an excess decay estimate, where the{ \it{excess}} function
is
\begin{align}\label{excess}
  \mathcal{\Phi}_s(B, \bfu):=\bigg(\dashint_{B}|\bfV(\nabla \bfu)-
  \mean{\bfV(\nabla \bfu)}_B|^{2s} \,dx\bigg)^{\frac 1s}
\end{align}
with $\bfV(\bfQ)=\sqrt{\frac{\phi'(|\bfQ|)}{|\bfQ|}}{\bfQ}$ and $s
\geq 1$. We write $\Phi:= \Phi_1$. Note that $\Phi_{s_1}(B, \bfu) \leq
\Phi_{s_2}(B,\bfu)$ for $1 \leq s_1 \leq s_2$ and $\abs{\bfV(\bfQ)}^2
\sim \phi(\abs{\bfQ})$.

Our regularity theorem states:
\begin{theorem}[Main theorem]
  \label{thm:main}
  Let $\bfu$ be a local minimizer of the quasiconvex
  functional~\eqref{funct}, with $f$ satisfying
  \ref{itm:freg}--\ref{itm:D2fdiff} and fix some $\beta\in (0,1)$.
  Then there exists $\delta=\delta(\beta)>0$ such that the following
  holds: If
  \begin{align}
    \label{eq:speed}
    \Phi(2B, \bfu) &\leq \delta \dashint_{2B}
    \abs{\bfV(\nabla \bfu)}^2\,dx
  \end{align}
  for some ball $B \subset \Rn$ with $2B \subset \Omega$, then
  $\bfV(\nabla\bfu)$ is $\beta$-H{\"o}lder continuous on~$B$.  
\end{theorem}
The proof of this theorem can be found at the end of
Section~\ref{sec:comparison}. We define the set of regular points $\mathcal{R} (\bfu)$ by
\begin{align}
  \mathcal { R} (\bfu)=\bigset{ x_0\in \Omega :\liminf_{r\to 0}
  \Phi(B(x_0,r), \bfu)=0}.
\end{align}
As an immediate consequence of Theorem~\ref{thm:main} we have:
\begin{corollary}
  \label{cor:main}
  Let $\bfu$ be as in Theorem~\ref{thm:main} and let $x_0 \in
  \mathcal{R}(\bfu)$ with $\nabla \bfu\not=0$. Then for every
  $\beta\in(0,1)$ the function
  $\bfV(\nabla \bfu)$ is $\beta$-H{\"o}lder continuous on a neighborhood
  of~$x_0$. 
\end{corollary}
Note that the H{\"o}lder continuity of $\bfV(\nabla \bfu)$ implies the
H{\"o}lder continuity of $\nabla \bfu$ with a different exponent depending
on ~$\phi$. Consider for example the situation $\phi(t) =
t^p$ with $1<p< \infty$. Therefore, $\beta$-H{\"o}lder continuity of
$\bfV(\nabla \bfu)$ implies for $p \leq 2$ that $\nabla \bfu$ is
$\beta$-H{\"o}lder continuous and for $p > 2$ that $\nabla \bfu$ is
$\beta \frac{2}{p}$-H{\"o}lder continuous.

The proofs of the regularity results for local minimizers
in~\cite{Eva86},\cite{AceF84}, \cite{Carfusmin98}, are based on a
blow-up technique originally developed by De Giorgi \cite{DeG61} and
Almgren \cite{Alm68}, \cite{Alm76} in the setting of the geometric
measure theory, and by Giusti and Miranda for elliptic systems,
\cite{GiuMir68}.

Another more recent approach for proving partial regularity for local
minimizers is based on the so called $\mathcal{A}$-harmonic
approximation method.  This technique has its origin in Simon's proof
of the regularity theorem \cite{Sim96} (see also Allard \cite{All72}).
The technique has been successfully applied in the framework of the
geometric measure theory, and to obtain partial-regularity results for
general elliptic systems  in a series of papers by Duzaar, Grotowski,
Kronz, Mingione ~\cite{DuzGro00} \cite{DuzGK05}, \cite{DuzM04},
\cite{DuzMin09h}. More precisely, we consider a bilinear form on
$\Hom(\Rn,\RN )$ which is (strongly) elliptic in the sense of
Legendre-Hadamard, i.e.  if for all $\bfa\in \RN, \bfb \in \Rn$ it
holds
\begin{align*}
  \mathcal{A}_{ij}^{\alpha\beta} a^i b_\alpha a^j b_\beta &\geq
  \kappaA \abs{\bfa}^2 \abs{\bfb}^2
\end{align*}
for some $\kappaA>0$. The method of $\mathcal{A}$-harmonic
approximation consists in obtaining a good approximation of functions
$\bfu\in W^{1,2}(B)$, which are {\em almost $\mathcal{A}$-harmonic}
(in the sense of Theorem~\ref{thm:Aappr_phi}) by
$\mathcal{A}$-harmonic functions $\bfh\in W^{1,2}(B)$, in both the
$L^2$-topology and in the weak topology of $W^{1,2}$. Let us recall that $\bfh \in
W^{1,2}(B)$ is called $\mathcal{A}$-harmonic on B if 
\begin{align}
  \int_B \mathcal{A}(D\bfh, D\bfeta) \,dx = 0, \forall{ \bfeta} \in
  C^\infty_0(B)
\end{align}
holds.
Here, in order to prove the result, we will follow the second
approach.

As in the situations considered in the above-mentioned papers, the
required approximate $\mathcal{A}$-harmonicity of a local minimizer
$\bfu\in W^{1,\phi}(\Omega \setminus \Sigma)$ is a  consequence
of the minimizing property and of the {\em smallness} of the {\em
  excess}.

Next, having proven the $\mathcal{A}$-harmonic approximation lemma and
the corresponding approximate $\mathcal{A}$-harmonicity of the local
minimizer $\bfu$, the other steps are quite standard. We prove a
Caccioppoli-type inequality for minimizers $\bfu$ and thus we compare
$\bfu$ with the $\mathcal{A}$-harmonic approximation $\bfh$ to obtain,
via our Caccioppoli-type inequality, the desired excess decay
estimate.\par

Thus, the main difficulty is to establish a suitable version of the
$\mathcal{A}$-harmonic approximation lemma in this general setting.
However, let us point out that our $\mathcal{A}$-harmonic
approximation lemma differs also in the linear or $p$-growth situation
from the classical one in~\cite{DuzM04}.  Firstly, we use a direct
approach based on the Lipschitz truncation technique which requires no
contradiction argument. This allows for a precise control of the
constants, which will only depend on the $\Delta_2$-condition
for $\phi$ and its conjugate.In fact, we will
 apply the approximation lemma to the family of shifted
N-functions that inherit the same  $\Delta_2$ constants of $\phi$.  Secondly, we are
able to preserve the boundary values of our original function, so
$\bfu-\bfh$ is a valid test function.  Thirdly, we show that $\bfh$
and $\bfu$ are close with respect to the gradients rather than just
the functions. The main tools in the proof is a Lipschitz
approximation of the Sobolev functions as in~\cite{DieMS08,BreDieFuc12}.
However, since $\mathcal{A}$ is only strongly elliptic in the sense of
Legendre-Hadamard, we will not be able to apply the Lipschitz
truncation technique directly to our almost $\mathcal{A}$-harmonic
function $\bfu$. Instead, we need to use duality and apply the
Lipschitz truncation technique to the test functions.\par
Let us conclude by observing that here we are able to present a
unified approach for both cases: superquadratic and subquadratic
growth.

\section{Notation and preliminary results}
\label{sec:notation}

We use $c, C$ as generic constants, which may change from line to
line, but does not depend on the crucial quantities. Moreover we write
$f\sim g$ iff there exist constants $c,C>0$ such that $c\, f \le g\le
C\, f$.  For $w\in L^{1}_{\loc}(\Rn)$ and a ball $B \subset \Rn$ we
define
\begin{align}
  \mean{w}_B := \medint_{B}w(x)\,dx := \frac{1}{|B|}\int_B w(x)\,dx,
\end{align}
where $\abs{B}$ is the $n$-dimensional Lebesgue measure of~$B$. For
$\lambda>0$ we denote by $\lambda B$ the ball with the same center as $B$
but $\lambda$-times the radius.   For $U, \Omega \subset \Rn$ we write $U
\compactsubset \Omega$ if the closure of~$U$ is a compact subset
of~$\Omega$.

The following definitions and results are standard in the context of
N-functions, see for example \cite{KraR61,RaoR91}.
A real function $\phi \,:\, \setR^{\geq 0} \to
\setR^{\geq 0}$ is said to be an N-function if it satisfies the
following conditions: $\phi(0)=0$ and there exists the derivative
$\phi'$ of $\phi$.  This derivative is right continuous,
non-decreasing and satisfies $\phi'(0) = 0$, $\phi'(t)>0$ for $t>0$,
and $\lim_{t\to \infty} \phi'(t)=\infty$. Especially, $\phi$ is
convex.

We say that $\phi$ satisfies the $\Delta_2$-condition, if there
exists $c > 0$ such that for all $t \geq 0$ holds $\phi(2t) \leq c\,
\phi(t)$. We denote the smallest possible constant by
$\Delta_2(\phi)$. Since $\phi(t) \leq \phi(2t)$ the $\Delta_2$
condition is equivalent to $\phi(2t) \sim \phi(t)$.

By $L^\phi$ and $W^{1,\phi}$ we denote the classical Orlicz and
Sobolev-Orlicz spaces, i.\,e.\ $f \in L^\phi$ iff $\int
\phi(\abs{f})\,dx < \infty$ and $f \in W^{1,\phi}$ iff $f, \nabla f
\in L^\phi$. By $W^{1,\phi}_0(\Omega)$ we denote the closure of
$C^\infty_0(\Omega)$ in $W^{1,\phi}(\Omega)$.

By $(\phi')^{-1} \,:\, \setR^{\geq 0} \to \setR^{\geq
  0}$ we denote the function
\begin{align*}
  (\phi')^{-1}(t) &:= \sup \set{ s \in \setR^{\geq 0}\,:\,
    \phi'(s) \leq t}.
\end{align*}
If $\phi'$ is strictly increasing then $(\phi')^{-1}$ is the inverse
function of $\phi'$.  Then $\phi^\ast \,:\, \setR^{\geq 0} \to
\setR^{\geq 0}$ with
\begin{align*}
  \phi^\ast(t) &:= \int_0^t (\phi')^{-1}(s)\,ds
\end{align*}
is again an N-function and $(\phi^\ast)'(t) =
(\phi')^{-1}(t)$ for $t>0$. It is the complementary function of
$\phi$.  Note that $\phi^*(t)= \sup_{s \geq 0} (st - \phi(s))$ and
$(\phi^\ast)^\ast = \phi$. For all $\delta>0$ there exists
$c_\delta$ (only depending on $\Delta_2({\phi, \phi^\ast})$
such that for all $t, s \geq 0$ holds
\begin{align}  \label{eq:young}
  t\,s &\leq \delta\, \phi(t) + c_\delta\, \phi^\ast(s),
\end{align}

For $\delta=1$ we have $c_\delta=1$.  This inequality is called {\em
  Young's inequality}. For all $t\geq 0$
\begin{gather}
  \label{ineq:phiast_phi_p_pre}
  \begin{aligned}
    \frac{t}{2} \phi'\Big(\frac{t}{2} \Big) \leq \phi(t)
    \leq t\,\phi'(t),
    \\
    \phi \bigg(\frac{\phi^\ast(t)}{t} \bigg) \leq \phi^\ast(t) \leq \phi
    \bigg( \frac{2\, \phi^\ast(t)}{t} \bigg).
  \end{aligned}
\end{gather}
Therefore, uniformly in $t\geq 0$
\begin{gather}
  \label{ineq:phiast_phi_p}
  \phi(t) \sim \phi'(t)\,t, \qquad
  \phi^\ast\big( \phi'(t) \big) \sim \phi(t),
\end{gather}
where the constants only depend on $\Delta_2(\phi, \phi^\ast)$.

We say that a  N-function $\psi$ is of type $(p_0,p_1)$ with $1 \leq
p_0 \leq p_1 < \infty$, if
\begin{align}
  \label{eq:typepq}
  \psi(st) \leq C\, \max \set{s^{p_0}, s^{p_1}} \psi(t) \qquad
  \text{for all $s,t \geq 0$}.
\end{align}
We also write $\psi \in \frT(p_0,p_1,C)$.
\begin{lemma}
  \label{lem:typepq}
  Let $\psi$ be an N-function with $\psi \in \Delta_2$ together with its conjugate.
  Then $\psi \in \frT(p_0,p_1,C_1)$ for some $1 < p_0 < p_1 <
  \infty$ and $C_1>0$, where $p_0$, $p_1$ and $C_1$ only depend on
  $\Delta_2(\psi, \psi^*)$.
  Moreover, $\psi$ has the representation
  \begin{align}\label{quasiconcave}
    \psi(t)=t^{p_0} \big(h(t)\big)^{p_1-p_0} \qquad \text{for all
      $t\geq 0$},
  \end{align}
  where $h$ is a quasi-concave
  function, i.e.
  \begin{align*}
    h(\lambda t) \leq  C_2\max\set{1,\lambda} h(t) \qquad \text{for all
      $\lambda,t \geq 0$,}
  \end{align*}
  where $C_2$ only depends on $\Delta_2(\psi,\psi^*)$.
\end{lemma}
\begin{proof}
  Let $K := \Delta_2(\psi)$ and $K_* := \max
  \set{\Delta_2(\psi^*),3}$. Then $\psi^*(2t) \leq K_* \psi^*(t)$ for
  all $t\geq 0$ implies $\psi(t) \leq K_* \psi(2t/K_*)$ for all $t\geq
  0$.  Now, choose $p_0, p_1$ such that $1< p_0 < p_1 < \infty$ and
  $K \leq 2^{p_0}$ and $(K_*/2)^{p_0} \leq K_*$.  We claim that
  \begin{align}
    \label{eq:typepq2}
    \psi(st) \leq C\, \max \set{s^{p_0}, s^{p_1}} \psi(t) \qquad
    \text{for all $s,t \geq 0$},
  \end{align}
  where $C$ only depends on $K$ and $K_*$. Indeed, if $s \geq 1$, then
  choose $m \geq 0$ such that $2^m \leq s \leq 2^{m+1}$. Using $\psi
  \in \Delta_2$, we get
  \begin{align}
    \label{eq:typepq3}
    \psi(st) \leq \psi(2^{m+1} t) \leq K^{m+1} \psi(t) \leq K
    (2^{p_1})^m \psi(t) \leq K s^{p_1} \psi(t).
  \end{align}
  If $s \leq 1$, then we choose $m \in \setN_0$ such that $(K_*/2)^m s \leq
  1 \leq (K_*/2)^{m+1} s$, so that
  \begin{align*}
    \psi(st) &\leq K_*^m \psi \Bigg(\bigg(\frac{2}{K_*}\bigg)^m
    st\Bigg) \leq K_* \bigg( \frac{K_*}{2}\bigg)^{p_0(m-1)} \psi(t)
    \leq K_* s^{p_0} \psi(t).
  \end{align*}
  This proves~\eqref{eq:typepq2}.

  Now, let us  define
  \begin{align*}
    h(u) := \psi \Big( u^{\frac{1}{p_1-p_0}} \Big)
    u^{-\frac{p_0}{p_1-p_0}},
  \end{align*}
  then $\psi$ satisfies~(\ref{quasiconcave}). It remains to show that
  $h$ is quasi-concave. We estimate with~\eqref{eq:typepq2}
  \begin{align*}
    h(su) \leq K\, \psi \Big( u^{\frac{1}{p_1-p_0}} \Big)
    \max\biggset{ s^{\frac{p_1}{p_1-p_0}}, s^{\frac{p_0}{p_1-p_0}} }
    (su)^{\frac{-p_0}{p_1-p_0}} = K \psi(u) \max \set{s,1}
  \end{align*}
  for all $s,u \geq 0$.
\end{proof}
Throughout the paper we will assume that  $\phi$ satisfies the
following assumption.
\begin{assumption}
  \label{ass:phi}
  Let $\phi$ be an N-function such that
  $\phi$ is $C^1$ on $[0,\infty)$ and $C^2$ on $(0,\infty)$. Further
  assume that
  \begin{align}
    \label{eq:phi_pp}
    \phi'(t) &\sim t\,\phi''(t)
  \end{align}
  uniformly in $t > 0$. The constants in~\eqref{eq:phi_pp} are called
  the {\em characteristics of~$\phi$}.
\end{assumption}
We remark that under these assumptions $\Delta_2({\phi,\phi^\ast})
< \infty$ will be automatically satisfied, where
$\Delta_2({\phi,\phi^*})$ depends only on the characteristics
of~$\phi$. 

For given $\phi$ we define the associated N-function $\psi$ by
\begin{align}
  \label{eq:def_psi}
  \psi'(t) &:= \sqrt{ \phi'(t)\,t\,}.
\end{align}

It is shown in~\cite[Lemma 25]{DieE08} that if $\phi$ satisfies
Assumption~\ref{ass:phi}, then also $\phi^*$, $\psi$, and $\psi^*$
satisfy this assumption.

Define $\bfA,\bfV\,:\, \setR^{N \times n} \to \setR^{N \times n}$ in
the following way:
\begin{subequations}
  \label{eq:defAV}
  \begin{align}
    \label{eq:defA}
    \bfA(\bfQ)&=\phi'(|\bfQ|)\frac{\bfQ}{|\bfQ|},
    \\
    \label{eq:defV}
    \bfV(\bfQ)&=\psi'(|\bfQ|)\frac{\bfQ}{|\bfQ|}.
  \end{align}
\end{subequations}
Another important set of tools are the {\rm shifted N-functions}
$\set{\phi_a}_{a \ge 0}$ introduced in~\cite{DieE08}, see
also~\cite{DieK08,RuzDie07}. We define for $t\geq0$
\begin{align}
  \label{eq:phi_shifted}
  \phi_a(t):= \int _0^t \varphi_a'(s)\, ds\qquad\text{with }\quad
  \phi'_a(t):=\phi'(a+t)\frac {t}{a+t}.
\end{align}
Note that $\phi_a(t) \sim \phi'_a(t)\,t$. Moreover, for $t \geq a$ we
have $\phi_a(t) \sim \phi(t)$ and for $t \leq a$ we have $\phi_a(t)
\sim \phi''(a) t^2$.  This implies that
$\phi_a(s\,t) \leq c\, s^2 \phi_a(t)$ for all $s \in [0,1]$, $a \geq
0$ and $t \in [0,a]$. The families $\set{\phi_a}_{a \ge 0}$ and
$\set{(\phi_a)^*}_{a \ge 0}$ satisfy the $\Delta_2$-condition uniformly in $a \ge 0$. 

The connection between $\bfA$, $\bfV$  and the shifted  functions of $\phi$
 is best reflected in the
following lemma \cite[Lemma~2.4]{DieSV09}, see also~\cite{DieE08}.
\begin{lemma}
  \label{lem:hammer}
  Let $\phi$ satisfy Assumption~\ref{ass:phi} and let $\bfA$
  and $\bfV$ be defined by~\eqref{eq:defAV}. Then
  \begin{align*}
    \big({\bfA}(\bfP) - {\bfA}(\bfQ)\big) \cdot \big(\bfP-\bfQ \big)
    &\sim \bigabs{ \bfV(\bfP) - \bfV(\bfQ)}^2 \sim
    \phi_{\abs{\bfP}}(\abs{\bfP - \bfQ}),
    \\
    \bigabs{{\bfA}(\bfP) - {\bfA}(\bfQ)} &\sim
    \phi_{\abs{\bfP}}'(\abs{\bfP - \bfQ}),
    \\
    \intertext{uniformly in $\bfP, \bfQ \in \setR^{N \times n}$ .
      Moreover,} \bfA(\bfQ) \cdot \bfQ \sim \abs{\bfV(\bfQ)}^2 &\sim
    \phi(\abs{\bfQ}),
  \end{align*}
  uniformly in $\bfQ \in \setR^{N \times n}$.
\end{lemma}
We state a generalization of Lemma~2.1 in~\cite{AceF84} to the context
of convex functions $\phi$.
\begin{lemma}[Lemma~20, \cite{DieE08}]
  \label{lem:phi_l_prop}
  Let $\phi$ be an N-function with $\Delta_2({\phi,
    \phi^\ast}) < \infty$.  Then uniformly for all $\bfP_0, \bfP_1
  \in \setR^{N \times n}$ with $\abs{\bfP_0}+\abs{\bfP_1}>0$ holds
  \begin{align}
    \label{eq:3}
    \int_0^1 \frac{\phi'( \abs{\bfP_\theta} )}{\abs{\bfP_\theta}}\,
    d\theta &\sim \frac{\phi'(\abs{\bfP_0} +
      \abs{\bfP_1})}{\abs{\bfP_0} + \abs{\bfP_1}},
  \end{align}
  where $\bfP_\theta := (1-\theta)\bfP_0 + \theta \bfP_1$. The
  constants only depend on $\Delta_2({\phi,\phi^\ast})$.
\end{lemma}
Note that~\ref{itm:D2fdiff} and the previous Lemma imply that
\begin{align}
  \label{eq:diffDf}
  \begin{aligned}
    \bigabs{(Df)(\bfQ)-(Df)(\bfP)} &= \biggabs{ \int_0^1 (D^2 f)(\bfP
      + t (\bfQ -\bfP)) (\bfQ - \bfP) \,dt}
    \\
    &\leq c\, \int_0^1 \phi''(\abs{\bfP + t (\bfQ -\bfP))}) \,dt
    \abs{\bfP - \bfQ}
    \\
    &\leq c\, \phi''(\abs{\bfP} + \abs{\bfQ}) \abs{\bfP - \bfQ}
    \\
    &\leq c\, \phi_{\abs{\bfQ}}'(\abs{\bfP-\bfQ}).
  \end{aligned}
\end{align}

The following version of Sobolev-Poincar{\'e} inequality can be found in~\cite[Lemma
7]{DieE08}.
\begin{theorem}[Sobolev-Poincar{\'e}]
  \label{thm:poincare}
  Let $\varphi$ be an N-function with $\Delta_2({\varphi,
    \varphi^\ast}) < \infty$.  Then there exist $0 < \alpha < 1$ and
  $K>0$ such that the following holds. If $B \subset \setR^n$ is some
  ball with radius~$R$ and $\bfw \in W^{1,\varphi}(B,\RN)$, then
  \begin{align}
    \label{eq:poincare} \dashint_B \varphi\bigg( \frac{\abs{\bfw -
        \mean{\bfw}_B}}{R} \bigg) \,dx \leq K\, \bigg( \dashint_B
    \varphi^{\alpha}(\abs{\nabla \bfw}) \,dx
    \bigg)^\frac{1}{\alpha},
  \end{align}
  where $\mean{\bfw}_B := \dashint_B \bfw(x)\, \,dx$.
\end{theorem}

\section{Caccioppoli estimate}
\label{sec:caccioppoli-estimate}

We need the following simple modification of  lemma 3.1, (Chap. 5) from \cite{Gia82}.
\begin{lemma}
  \label{lem:giamod}
  Let $\psi$ be an N-function with $\psi \in \Delta_2$, let $r>0$ and
 $h \in L^\psi(B_{2r}(x_0))$. Further, let $f\,:\, [r/2,r]
  \to [0, \infty)$ be a bounded function such that for all
  $\frac{r}{2} < s < t < r$
  \begin{align*}
    f(s) \leq \theta f(t) + A \int_{B_t(x_0)} \psi\bigg(
    \frac{\abs{h(y)}}{t-s} \bigg) \,dy
  \end{align*}
  where $A> 0$ and $\theta \in [0,1)$. Then
  \begin{align*}
    f\bigg(\frac{r}{2}\bigg) \leq c(\theta,\Delta_2(\psi))\, A
    \!\!\!\!  \int_{B_{2r}(x_0)} \psi\bigg(\frac{\abs{h(y)}}{2r}\bigg)
    \,dy.
  \end{align*}
\end{lemma}
\begin{proof}
  Since $\psi \in \Delta_2$, there exists $C_2>0$ and $p_1 < \infty$
  (both depending only on $\Delta_2(\psi)$) such that $\psi(\lambda u)
  \leq C_2 \lambda^{p_1} \psi(u)$ for all $\lambda \geq 1$ and $u \geq
  0$ (compare~\eqref{eq:typepq3} of Lemma~\ref{lem:typepq}).  This
  implies
  \begin{align*}
    f(t) &\leq \theta f(s) + \underbrace{A\int_{B_s(x_0)}
      \psi\bigg( \frac{\abs{h(y)}}{2r} \bigg) \,dy\; C_2
      (2r)^{p_1}}_{=:A}\, (s-t)^{-p_1}.
  \end{align*}
  Now Lemma~3.1 in  \cite{Gia82}, with $\alpha := p_1$ implies
  \begin{align*}
    f\bigg(\frac{r}{2}\bigg) \leq c(\theta,p_1) A \int_{B_s(x_0)}
      \psi\bigg( \frac{\abs{h(y)}}{2r} \bigg) \,dy\; C_2
      (2r)^{p_1} r^{-p_1},
  \end{align*}
  which proves the claim.
\end{proof}

\begin{theorem}
  \label{thm:cacc}
 Let $\bfu\in W^{1,\phi}_{loc}(\Omega)$ be a local minimizer of $\mathcal{F}$ and
 $B$ be a ball with radius~$R$ such that  $2B\subset  \subset \Omega$. Then 
    \begin{align*}
      \int_{B} \phi_{\abs{\bfQ}}(\abs{\nabla \bfu - \bfQ})\,dx \leq c\,
      \int_{2B} \phi_{\abs{\bfQ}}\bigg( \frac{\abs{\bfu - \bfq}}{R}
      \bigg) \,dx
  \end{align*}
  for all $\bfQ \in \RNn$ and all linear polynomials $\bfq$ on $\Rn$
  with values in $\RN$ and $\nabla\bfq=\bfQ$, where $c$ only depends
  on $n$, $N$, $k$, $K$ and the characteristics of $\phi$.
\end{theorem}
\begin{proof}
  Let $0 < s < t$. Further, let $B_s$ and $B_t$ be balls in $\Omega$
  with the same center and with radius $s$ and $t$, respectively.
  Choose $\eta \in C^\infty_0(B_t)$ with $\chi_{B_s} \leq \eta \leq
  \chi_{B_t}$ and $\abs{\nabla \eta} \leq c/(t-s)$. Now,
  define $\bfxi := \eta (\bfu - \bfq)$ and $\bfz := (1-\eta)(\bfu -
  \bfq)$.  Then $\nabla \bfxi + \nabla \bfz = \nabla \bfu - \bfQ$.
  Consider
  \begin{align*}
    I := \int_{B_t} f(\bfQ + \nabla \bfxi) - f(\bfQ) \,dx.
  \end{align*}
  Then by the quasi-convexity of~$f$, see~\ref{itm:fquasi}, follows
  \begin{align*}
    I &\geq c \int_{B_t} \phi_{\abs{\bfQ}}(\abs{\nabla \bfxi})\,dx.
  \end{align*}
  On the other hand since $\nabla \bfxi + \nabla \bfz = \nabla \bfu -
  \bfQ$ we get
  \begin{align*}
    I &= \int_{B_t} f(\bfQ + \nabla \bfxi) - f(\bfQ) \,dx
    \\
    &= \int_{B_t} f(\bfQ + \nabla \bfxi) - f(\bfQ + \nabla \bfxi +
    \nabla \bfz) \,dx
    \\
    &\quad + \int_{B_t} f(\nabla \bfu) - f(\nabla \bfu - \nabla \bfxi)
    \,dx
    \\
    &\quad + \int_{B_t} f(\bfQ + \nabla \bfz) - f(\bfQ) \,dx
    \\
    &=: II + III + IV.
  \end{align*}
  Since $\bfu$ is a local minimizer, we know that $(III) \leq
  0$. Moreover,
  \begin{align*}
    II + IV &= \int_{B_t} \int_0^1 \big((Df)(\bfQ + t \nabla \bfz)
    -(Df)(\bfQ + \nabla \bfxi - t \nabla \bfz)\big) \nabla \bfz \,dt
    \,dx
    \\
    &= \int_{B_t} \int_0^1 \big((Df)(\bfQ + t \nabla \bfz) -(Df)(\bfQ)
    \nabla \bfz \,dt \,dx
    \\
    &\quad - \int_{B_t} \int_0^1 \big((Df)(\bfQ + \nabla \bfxi - t
    \nabla \bfz)-(Df)(\bfQ) \big) \nabla \bfz \,dt \,dx .
  \end{align*}
  This proves
  \begin{align*}
    |II| + |IV| &\leq c \int_{B_t} \int_0^1
    \phi'_{\abs{\bfQ}}(t\abs{\nabla \bfz})\,dt \abs{\nabla \bfz} \,dx
    \\
    &\quad + c \int_{B_t} \int_0^1 \phi'_{\abs{\bfQ}}(\abs{\nabla \bfxi
      - t \nabla \bfz}) \,dt \abs{\nabla \bfz} \,dx
  \end{align*}
  Using $\phi'_{\abs{\bfQ}}(\abs{\nabla \bfxi
    - t \nabla \bfz}) \leq c\, \phi'_{\abs{\bfQ}}(\abs{\nabla \bfxi})
  + c\, \phi'_{\abs{\bfQ}}(\abs{\bfz})$, we get
  \begin{align*}
    |II| + |IV| &\leq c \int_{B_t} \phi_{\abs{\bfQ}}(\abs{\nabla \bfz})
    \,dx + c \int_{B_t} \phi'_{\abs{\bfQ}}(\abs{\nabla \bfxi })
    \abs{\nabla \bfz} \,dx
    \\
    &\leq c \int_{B_t} \phi_{\abs{\bfQ}}(\abs{\nabla \bfz}) \,dx +
    \frac{1}{2} (I),
  \end{align*}
  where we have used Young's inequality in the last step. Overall, we
  have shown the a~priori estimate
  \begin{align}
    \label{eq:capaux1}
    \int_{B_t} \phi_{\abs{\bfQ}}(\abs{\nabla \bfxi})\,dx \leq c\,
    \int_{B_t} \phi_{\abs{\bfQ}}(\abs{\nabla \bfz}) \,dx.
  \end{align}
  Note that $\nabla \bfz = (1-\eta) (\nabla \bfu - \bfQ) - \nabla \eta
  (\bfu - \bfq)$, which is zero outside $B_t \setminus B_s$. Hence,
  \begin{align*}
    \int_{B_t} \phi_{\abs{\bfQ}}(\abs{\nabla \bfxi})\,dx \leq c\,
    \int_{B_t\setminus B_s} \phi_{\abs{\bfQ}}(\abs{\nabla \bfu - \bfQ})
    \,dx + c\, \int_{B_t} \phi_{\abs{\bfQ}}\bigg( \frac{\abs{\bfu -
        \bfq}}{t-s} \bigg) \,dx.
  \end{align*}
  Since $\eta = 1$ on $B_s$, we get
  \begin{align*}
    \int_{B_s} \phi_{\abs{\bfQ}}(\abs{\nabla \bfu - \bfQ})\,dx \leq
    c\, \int_{B_t\setminus B_s} \phi_{\abs{\bfQ}}(\abs{\nabla \bfu -
      \bfQ}) \,dx + c\, \int_{B_t} \phi_{\abs{\bfQ}}\bigg(
    \frac{\abs{\bfu - \bfq}}{t-s} \bigg) \,dx.
  \end{align*}
  The hole-filling technique proves
  \begin{align*}
    \int_{B_s} \phi_{\abs{\bfQ}}(\abs{\nabla \bfu - \bfQ})\,dx \leq \lambda
    \int_{B_t} \phi_{\abs{\bfQ}}(\abs{\nabla \bfu - \bfQ}) \,dx + c\,
    \int_{B_t} \phi_{\abs{\bfQ}}\bigg( \frac{\abs{\bfu - \bfq}}{t-s}
    \bigg) \,dx
  \end{align*}
  for some $\lambda \in (0,1)$, which is independent of $\bfQ$ and
  $\bfq$. Now Lemma  \ref{lem:giamod} proves the claim.
\end{proof}

\begin{corollary}
  \label{cor:gehring2} There exists $0<\alpha<1$ such that for all
  local minimizers $\bfu\in W^{1,\phi}_{loc}(\Omega)$ of
  $\mathcal{F}$, all balls $B$ with $2B \subset \subset \Omega$ and
  all~$\bfQ \in\RNn$
  \begin{align*}
    \dashint_B \abs{\bfV(\nabla \bfu) -
      \bfV(\bfQ)}^2\,dx &\leq
    c\, \bigg(\dashint_{2B} \abs{\bfV(\nabla \bfu) -
\bfV(\bfQ)}^{2\alpha}\,dx\bigg)^{\frac{1}{\alpha}}
  \end{align*}
\end{corollary}
\begin{proof}
Apply Theorem \ref{thm:cacc} with $\bfq$ such that $\langle \bfu - \bfq\rangle_{2B}=0$. Then use
Theorem \ref{thm:poincare} with $\bfw(x)=\bfu(x)-\bfQ x$.
\end{proof}

Using Gehring's Lemma we deduce the following assertion.
\begin{corollary}
  \label{cor:gehring3} 
  There exists $s_0>1$ such that for all local minimizers
  $\bfu\in W^{1,\phi}_{loc}(\Omega)$ of $\mathcal{F}$, all balls $B$
  with $2B \subset \subset \Omega$ and all~$\bfQ \in\RNn$
  \begin{align*}
    \bigg(\dashint_B \abs{\bfV(\nabla \bfu) -
      \bfV(\bfQ)}^{2s_0}\,dx\bigg)^\frac{1}{s_0} &\leq c\,
    \dashint_{2B} \abs{\bfV(\nabla \bfu) - \bfV(\bfQ)}^{2}\,dx.
  \end{align*}
\end{corollary}

\section{\texorpdfstring{The $\mathcal{A}$-harmonic approximation}{The
    A-harmonic approximation}}
\label{sec:mathcala-harmonicity}

In this section we present a generalization of the
$\mathcal{A}$-harmonic approximation lemma in Orlicz spaces. 
Basically it  says that if a function locally ``almost'' behaves
like an $\mathcal{A}$-harmonic function, then it is close to an
$\mathcal{A}$-harmonic function. The proof is based on the Lipschitz
truncation technique, which goes back to Acerbi- Fusco~\cite{AceF84}
but has been refined by many others.

Orginally the closeness of the function to its $\mathcal{A}$-harmonic
approximation was stated in terms of the $L^2$-distance and later for
the non-linear problems in terms of the $L^p$-distance. Based on a
refinement of the Lipschitz truncation technique~\cite{DieMS08}, it
has been shown in~\cite{DieStrVer10} that also the distance in terms
of the gradients is small.

Let us consider the following elliptic system
\begin{align*}
  - \partial_\alpha( \mathcal{A}_{ij}^{\alpha\beta} D_\beta u^j) &=
  -\partial_\alpha H_i^\alpha &\qquad& \text{in $B$},
\end{align*}
where $\alpha,\beta=1, \dots,n$ and $i,j=1, \dots N$. We use the
convention that repeated indices are summed. In short we write
$-\divergence(\mathcal{A} \nabla \bfu)= -\divergence\bfG$. We assume
that $\mathcal{A}$ is  constant.
We say that $\mathcal{A}$ is {\em strongly elliptic in
  the sense of Legendre-Hadamard} if for all $\bfa\in \RN, \bfb \in
\Rn$ holds
\begin{align*}
  \mathcal{A}_{ij}^{\alpha\beta} a^i b_\alpha a^j b_\beta &\geq
  \kappaA \abs{\bfa}^2 \abs{\bfb}^2
\end{align*}
for some $\kappaA>0$. The biggest possible constant~$\kappaA$ is
called the ellipticity constant of~$\mathcal{A}$. By
$\abs{\mathcal{A}}$ we denote the Euclidean norm of~$\mathcal{A}$.  We
say that a Sobolev function $\bfw$ on a ball~$B$ is
$\mathcal{A}$-harmonic, if it satisfies $-\divergence (\mathcal{A}
\nabla \bfw)=0$ in the sense of distributions.

Given a Sobolev function $\bfu$ on a ball~$B$ we want to find an
$\mathcal{A}$-harmonic function~$\bfh$ which is close the our
function~$\bfu$. The way  to find ~$\bfh$ is very simple: it
will be the $\mathcal{A}$-harmonic function with the same boundary
values as $\bfu$. In particular, we want to find a Sobolev function~$\bfh$
which satisfies
\begin{alignat}{2}
  \label{eq:calA1}
  \begin{aligned}
    -\divergence (\mathcal{A} \nabla \bfh) &= 0 &\qquad&\text{on $B$}
    \\
    \bfh&= \bfu &\qquad&\text{on $\partial B$}
  \end{aligned}
\end{alignat}
in the sense of distributions.

Let $\bfw := \bfh - \bfu$, then~\eqref{eq:calA1} is equivalent to finding
a Sobolev function~$\bfw$ which satisfies
\begin{alignat}{2}
  \label{eq:calA2}
  \begin{aligned}
    -\divergence (\mathcal{A} \nabla \bfw) &= -\divergence(\mathcal{A}
    \nabla \bfu) &\qquad&\text{on $B$}
    \\
    \bfw&= \bfzero &\qquad&\text{on $\partial B$}
  \end{aligned}
\end{alignat}
in the sense of distributions.  

Our main approximation result is the following.
\begin{theorem}
  \label{thm:Aappr_phi}
  Let $B \subset \subset \Omega$ be a ball with radius~$r_B$ and let
  $\widetilde{B} \subset \Omega$ denote either~$B$ or $2B$. Let
  $\mathcal{A}$ be strongly elliptic in the sense of
  Legendre-Hadamard.  Let $\psi$ be an N-function with $
  \Delta_2(\psi,\psi^*)<\infty$ and let $s>1$. Then for every $\epsilon>0$,
  there exists $\delta>0$ only depending on $n$, $N$, $\kappa_A$,
  $\abs{\mathcal{A}}$, $\Delta_2(\psi,\psi^*)$ and~$s$ such
  that the following holds: let $\bfu \in W^{1,\psi}(\widetilde{B})$
  be {\em almost $\mathcal{A}$-harmonic} on~$B$ in the sense that
  \begin{align}
    \label{eq:Aappr_ah}
    \biggabs{\dashint_B \mathcal{A}\nabla \bfu \cdot \nabla \bfxi\,dx}
    \leq \delta \dashint_{\widetilde{B}} \abs{\nabla \bfu}\,dx\,
    \norm{\nabla \bfxi}_{L^\infty(B)}
  \end{align}
  for all $\bfxi \in C^\infty_0(B)$. Then the unique solution $\bfw
  \in W^{1, \psi}_0(B)$ of~\eqref{eq:calA2} satisfies
  \begin{align}
    \label{eq:Aappr_est}
    \dashint_B \psi\bigg(\frac{\abs{\bfw}}{r_B}\bigg)\,dx + \dashint_B
    \psi(\abs{\nabla \bfw})\,dx \leq \epsilon \Bigg(
    \bigg(\dashint_{B} \big(\psi(\abs{\nabla \bfu})\big)^s
    \,dx\bigg)^{\frac 1s} + \dashint_{\widetilde{B}} \psi(\abs{\nabla
      \bfu}) \,dx\Bigg).
  \end{align}
\end{theorem}
The proof of this theorem can be found at the end of this section.
The distinction between $B$ and $\tilde{B}$ on the right-hand side
of~\eqref{eq:Aappr_est} allows a finer tuning with respect to the
exponents. If $B=\tilde{B}$, then only the term involving~$s$ is
needed.

The following result on the solvability and uniqueness in the setting
of classical Sobolev spaces $W^{1,q}_0(B, \RN)$ can be found
in~\cite[Lemma~2]{DolM95}.
\begin{lemma}
  \label{lem:DolM}
  Let $B \subset\subset \Omega$ be a ball, let $\mathcal{A}$ be
  strongly elliptic in the sense of Legendre-Hadamard and let $1 < q <
  \infty$. Then for every $\bfG \in L^q(B, \RNn)$, there exists a
  unique weak solution $\bfu=\TA \bfG \in W^{1,q}_0(B, \RN)$ of
  \begin{align}
    \label{eq:divAdivG}
    \begin{aligned}
      -\divergence(\mathcal{A} \nabla \bfu) &= -\divergence \bfG
      &\qquad&\text{on B},
      \\
      \bfu&= \bfzero &\qquad&\text{on $\partial B$}.
    \end{aligned}
  \end{align}
  The solution operator~$\TA$ is linear and satisfies
  \begin{align*}
    \begin{aligned}
      \norm{\nabla \TA \bfG}_{L^q(B)} &\leq c\, \norm{\bfG}_{L^q(B)},
    \end{aligned}
  \end{align*}
  where $c$ only depends on $n$, $N$,
  $\kappaA$, $\abs{\mathcal{A}}$ and $q$.

\end{lemma}
\begin{remark}
 Note
  that our constants do not depend on the size of the ball, since the estimates involved
  are scaling invariant.
\end{remark}
Let $\TA$ be the solution operator of Lemma~\ref{lem:DolM}. Then by the
uniqueness of Lemma~\ref{lem:DolM}, the operator $\TA \,:\, L^q(B, \RNn)
\to W^{1,q}_0(B, \RN)$ does not depend on the choice of $q \in
(1,\infty)$. Therefore, $\TA$ is uniquely defined from $\bigcup_{1<q<
  \infty} L^q(B, \RNn)$ to $\bigcup_{1<q< \infty} W^{1,q}_0(B, \RN)$

We need to extend Lemma~\ref{lem:DolM} to the setting of Orlicz
spaces. We will do so by means of the following real interpolation
theorem of Peetre~\cite[Theorem~5.1]{Pee70} which 
states, that whenever $\psi$ is of the form~(\ref{quasiconcave}), then
$L^\psi$ is an interpolation space between $L^{p_0}$ and $L^{p_1}$.

\begin{theorem}
  \label{thm:peetre}
  Let $\psi$ be an N-function with $\Delta_2 (\psi, \psi^*)$
  and $p_0$, $p_1$ as in Lemma \ref{lem:typepq}. Moreover let $S$ be a
  linear, bounded operator from $L^{p_j} \to L^{p_j}$ for $j=0,1$.
  Then there exists $K_2$, which only depends on $\Delta_2(\psi,\psi^*)$,
   and the operator norms of $S$ such that
  \begin{align*}
    \norm{Sf}_\psi &\leq K_2 \norm{f}_\psi
    \\
    \int \psi(\abs{Sf}/K_2) \,d\mu &\leq \int \psi(\abs{f})\,d\mu
  \end{align*}
  for every $f \in L^\psi$.
\end{theorem}
This interpolation result and Lemma~\ref{lem:DolM} immediately imply:
\begin{theorem}
  \label{lem:DolMrho}
  Let $B \subset \Omega$ be a ball, let $\mathcal{A}$ be strongly
  elliptic in the sense of Legendre-Hadamard and $\psi$ be an
  N-function with $ \Delta_2 (\psi,\psi^*)$. Then the solution
  operator $\TA$ of Lemma~\ref{lem:DolM} is continuous from $L^\psi(B,
  \RNn)$ to $W^{1,\psi}_0(B, \Rn)$ and
  \begin{align}
    \label{eq:Testphi}
    \begin{aligned}
      \norm{\nabla \TA \bfG}_{L^\psi(B)} &\leq c\,
      \norm{\bfG}_{L^\psi(B)},
      \\
      \int_B \psi(\abs{\nabla \TA \bfG})\,dx &\leq c \int_B
      \psi(\abs{\bfG})\,dx,
    \end{aligned}
  \end{align}
  for all $\bfG \in L^\psi(B,\RNn)$, where $c$ only depends on $n$,
  $N$, $\kappaA$, $\abs{\mathcal{A}}$, $\Delta_2(\psi,\psi^*)$.
\end{theorem}
\begin{remark}
  \label{rem:DolMuniquephi}
  Since $\psi$ satisfies~\eqref{eq:typepq} for some $1<p_0<p_1 <
  \infty$ it follows easily that $L^\psi(B) \embedding L^{p_0}(B)$ for
  every ball $B \subset \Omega$. From this and the uniqueness in
  Lemma~\ref{lem:DolM}, the solution of~\eqref{eq:divAdivG} is also
  unique in $W^{1,\psi}_0(B, \RN)$.
\end{remark}
Since $\mathcal{A}$ is only strongly elliptic in the sense of
Legendre-Hadamard, we will not be able to apply the Lipschitz
truncation technique directly to our almost $\mathcal{A}$-harmonic
function $\bfu$. Instead, we need to use duality and apply the
Lipschitz truncation technique to the test functions. For this reason,
we prove the following variational inequality.
\begin{lemma}
  \label{lem:varineqphi}
  Let $B \subset \Omega$ be a ball and let $\mathcal{A}$ be strongly
  elliptic in the sense of Legendre-Hadamard.  Then it holds for all
  $\bfu \in W^{1, \psi}_0(B)$ that
  \begin{subequations}
    \begin{align}
      \label{eq:varest-norm}
      \norm{\nabla \bfu}_{\psi} &\sim \sup_{\substack{\bfxi \in
          C^\infty_0(B)\\ \norm{\nabla \bfxi}_{\psi^*} \leq 1}}
      \int_B \mathcal{A}\nabla \bfu \cdot \nabla \bfxi\,dx,
      \\
      \label{eq:varest-mod}
      \int_B \psi(\abs{\nabla \bfu})\,dx &\sim \sup_{\bfxi \in
        C^\infty_0(B)} \bigg[ \int_B \mathcal{A}\nabla \bfu \cdot
      \nabla \bfxi\,dx - \int_B \psi^*(\abs{\nabla \bfxi})\,dx \bigg].
    \end{align}
  \end{subequations}
  The implicit constants only depend on $n$, $N$, $\kappaA$,
  $\abs{\mathcal{A}}$, $\Delta_2(\psi,\psi^*)$.
\end{lemma}
\begin{proof}
  We begin with the proof of~\eqref{eq:varest-norm}.  The $\gtrsim$
  estimate is a simple consequence of H{\"o}lder's inequality, so let us
  concentrate on $\lesssim$. Since $(L^\psi)^* \cong L^{(\psi^*)}$
  (with constants bounded by~$2$) and $C^\infty_0(B)$ is dense in
  $L^{(\psi^*)}(\Omega)$, we have
  \begin{align*}
    \norm{\nabla \bfu}_{\psi} &\leq 2\, \sup_{\substack{\bfH \in
        C^\infty_0(B, \RNn)\\ \norm{\bfH}_{\psi^*} \leq 1}} \int_B
    \nabla \bfu \cdot \bfH \,dx,
  \end{align*}
  Define $\overline{\mathcal{A}}$ by
  $\overline{\mathcal{A}}_{ij}^{\alpha\beta} :=
  \mathcal{A}_{ji}^{\beta\alpha}$, then
  $-\divergence(\overline{\mathcal{A}} \nabla \bfu)$ is the formal
  adjoint operator of  $-\divergence({\mathcal{A}} \nabla \bfu).$  In particular, using~\eqref{eq:divAdivG}
  \begin{align}
    \label{eq:useTAbar}
    \begin{aligned}
      \int_B \nabla \bfu \cdot \bfH & = \int_B \nabla \bfu \cdot \overline{\mathcal{A}} \nabla
      \TAbar \bfH\,dx
      \\
      &= \int_B \mathcal{A} \nabla \bfu \cdot \nabla \TAbar \bfH\,dx.
    \end{aligned}
  \end{align}
  Hence,
  \begin{align*}
    \norm{\nabla \bfu}_{\psi} &\leq 2\, \sup_{\substack{\bfH \in
        C^\infty_0(B, \RNn)\\ \norm{\bfH}_{\psi^*} \leq 1}} \int_B
    \mathcal{A} \nabla \bfu \cdot \nabla \TAbar \bfH\,dx
    \\
    &\leq 4\, \sup_{\substack{\bfH \in C^\infty_0(B, \RNn)\\
        \norm{\bfH}_{\psi^*} \leq 1}} \, \norm{\mathcal{A}\nabla
      \bfu}_{L^\psi(B)} \norm{\nabla \TAbar \bfH}_{\psi^*}.
    \\
    &\leq c\, \, \norm{\mathcal{A}\nabla \bfu}_{L^\psi(B)},
  \end{align*}
  where we used in the last step Theorem~\ref{lem:DolMrho} for $\TAbar$
  and $\psi^*$. This proves~\eqref{eq:varest-norm}.

  Let us now prove~\eqref{eq:varest-mod}. The estimate $\gtrsim$ just
  follows from
  \begin{align*}
    \int_B \mathcal{A}\nabla \bfu \cdot \nabla \bfxi\,dx - \int_B
    \psi^*(\abs{\nabla \bfxi})\,dx &\leq \int_B \psi(\abs{\mathcal{A}}
    \abs{\nabla \bfu})\,dx
    \\
    &\leq c(\abs{\mathcal{A}}) \int_B \psi(\abs{\nabla \bfu})\,dx,
  \end{align*}
  where we used  $\abs{\mathcal{A} \nabla \bfu \cdot \nabla \bfxi} \leq
  \abs{\mathcal{A}} \abs{\nabla \bfu} \abs{\nabla \bfxi}$, Young's
  inequality and $\psi \in \Delta_2$.

  We turn to $\lesssim$ of~\eqref{eq:varest-mod}. Recall that
  \begin{align*}
    \psi^{**}(t) = \psi(t) = \sup_{u \geq 0} \big(ut - \psi^*(u)\big),
  \end{align*}
  where the supremum is attained at $u = \psi'(t)$.  Thus the choice
  $\bfH := \psi'(\abs{\nabla \bfu}) \frac{\nabla \bfu}{\abs{\nabla
      \bfu}}$ (with $\bfH=\bfzero$ where $\nabla \bfu=\bfzero$)
  implies
  \begin{align*}
    \int_B \psi(\abs{\nabla \bfu})\,dx &\leq \sup_{\bfH \in
      (L^{\psi^*}(B, \RNn))} \bigg[ \int_B \nabla \bfu
    \cdot \bfH\,dx - \int_B \psi^*(\abs{\bfH})\,dx \bigg].
  \end{align*}
  Using $\TAbar$ we estimate with~\eqref{eq:useTAbar}
  \begin{align*}
    \int_B \psi(\abs{\nabla \bfu})\,dx &\leq \sup_{\bfH \in
      L^{\psi^*}(B,\RNn)} \bigg[ \int_B \mathcal{A} \nabla
    \bfu \cdot \nabla \TAbar \bfH \,dx - \int_B
    \psi^*(\abs{\bfH})\,dx \bigg].
  \end{align*}
  By Theorem~\ref{lem:DolMrho} there exists $c \geq 1$ such that
  \begin{align*}
    \int_B
    \psi^*(\abs{\nabla \TA \bfH})\,dx \leq c\, \int_B
    \psi^*(\abs{\bfH})\,dx.
  \end{align*}
  This proves the following:
  \begin{align*}
    \int_B \psi(\abs{\nabla \bfu})\,dx &\leq \sup_{\bfH \in
      (L^{\psi^*}(B, \RNn)} \bigg[ \int_B \mathcal{A} \nabla \bfu
    \cdot \nabla \TAbar \bfH \,dx -c \int_B
    \psi^*(\abs{\nabla \TAbar \bfH})\,dx \bigg]
    \\
    &\leq \sup_{\bfxi \in L^{\psi^*}(B,\RN)} \bigg[ \int_B \mathcal{A}
    \nabla \bfu \cdot \nabla \bfxi \,dx - c \int_B
    \psi^*(\abs{\nabla \bfxi})\,dx \bigg].
  \end{align*}
  We replace $\bfu$ by
  $c \bfu$ to get
  \begin{align*}
    \int_B \psi\bigg(c \abs{\nabla \bfu}\bigg)\,dx &\leq
   c \sup_{\bfxi \in L^{\psi^*}(B,\RN)} \bigg[ \int_B
    \mathcal{A} \nabla \bfu \cdot \nabla \bfxi \,dx - \int_B
    \psi^*(\abs{\nabla \bfxi})\,dx \bigg].
  \end{align*}
  Now the claim follows using $\psi \in \Delta_2$ on the left-hand
  side and the density of $C^\infty_0(B, \RN)$ in $L^{\psi^*}(B, \RN)$
  (using $\psi^* \in \Delta_2$).
\end{proof}
Moreover, we need the following result
of~\cite[Theorem~3.3]{DieStrVer10} about {\em Lipschitz truncations}
in Orlicz spaces.
\begin{theorem}[Lipschitz truncation]
  \label{thm:Liptrun}
  Let $B \subset \Omega$ be a ball and let $\psi$ be an N-function with
  $\Delta_2 (\psi,\psi^*)<\infty$.  If $\bfw \in W^{1,\psi}_0(B,
  \RN)$, then for every $m_0 \in \setN$ and $\gamma>0$ there exists
  $\lambda \in [\gamma, 2^{m_0} \gamma]$ and $\bfw_\lambda \in
  W^{1,\infty}_0(B, \RN)$ (called the Lipschitz truncation) such that
  \begin{align*}
    \norm{\nabla \bfw_\lambda}_\infty &\leq c\, \lambda,
    \\
    \dashint_B \psi\big(\abs{\nabla \bfw_\lambda}
    \chi_{\set{\bfw_\lambda \not= \bfw}}\big) \,dx &\leq c\,
    \psi(\lambda) \frac{\abs{\set{\bfw_\lambda \not= \bfw}}}{\abs{B}}
    \leq \frac{c}{m_0} \dashint_B \psi(\abs{\nabla \bfw})\,dx
    \\
    \dashint_B \psi\big(\abs{\nabla \bfw_\lambda}\big) \,dx &\leq c\,
    \dashint_B \psi\big(\abs{\nabla \bfw}\big) \,dx.
  \end{align*}
  The constant $c$ depends only on $\Delta_2(\psi,\psi^*)$,
  $n$ and $N$.
\end{theorem}
\noindent
We are ready to prove  Theorem~\ref{thm:Aappr_phi}.
\begin{proof}[Proof of Theorem~\ref{thm:Aappr_phi}.]
  We begin with an application of Lemma~\ref{lem:varineqphi}:
  \begin{align}
    \label{eq:Liptrunvi}
    \dashint_B \psi(\abs{\nabla \bfu})\,dx \leq c\, \sup_{\bfxi \in
      C^\infty_0(B,\RN)} \bigg[ \dashint_B \mathcal{A}\nabla \bfu \cdot
    \nabla \bfxi\,dx - \dashint_B \psi^*(\abs{\nabla \bfxi})\,dx \bigg].
  \end{align}
 In the following let us fix $\bfxi \in C^\infty_0(B)$. Choose
  $\gamma\geq 0$ such that
  \begin{align}
    \label{eq:choicegamma}
    \psi^*(\gamma) &= \dashint_B \psi^*(\abs{\nabla \bfxi}) \,dx.
  \end{align}
  and let $m_0 \in \setN$. Due to Theorem~\ref{thm:Liptrun} applied
  to $\psi^*$ we find $\lambda \in [\gamma, 2^{m_0} \gamma]$ and
  $\bfxi_\lambda \in W^{1,\infty}_0(B)$ such that
  \begin{align}
    \label{eq:xi1}
    \norm{\nabla \bfxi_\lambda}_\infty &\leq c\, \lambda,
    \\
    \label{eq:xi2}
    \psi^*(\lambda) \frac{\abs{\set{\bfxi_\lambda \not=
          \bfxi}}}{\abs{B}} &\leq \frac{c}{m_0} \dashint_B
    \psi^*(\abs{\nabla \bfxi})\,dx
    \\
    \label{eq:xi3}
    \dashint_B \psi^*\big(\abs{\nabla \bfxi_\lambda}\big) \,dx &\leq c
    \dashint_B \psi^*\big(\abs{\nabla \bfxi}\big) \,dx.
  \end{align}
  Let us point out that the use of the Lipschitz truncation is not a
  problem of the regularity of~$\bfxi$ as it is $C^\infty_0$. It is
  the precise estimates above that we need.

  We calculate
  \begin{align*}
    \dashint_B \!\mathcal{A}\nabla \bfu \!\cdot\! \nabla \bfxi\,dx &=
    \dashint_B \!\mathcal{A}\nabla \bfu \!\cdot\! \nabla
    \bfxi_\lambda\,dx + \dashint_B \!\mathcal{A}\nabla \bfu \!\cdot\!
    \nabla (\bfxi-\bfxi_\lambda) \,dx =: I + II.
  \end{align*}
  Using Young's inequality and~\eqref{eq:xi3} we estimate
  \begin{align*}
    II&= \dashint_B \!\mathcal{A}\nabla \bfu \!\cdot\!  \nabla
    (\bfxi-\bfxi_\lambda) \chi_{\set{\bfxi \not= \bfxi_\lambda}} \,dx
    \\
    &\leq c\, \dashint_B \psi(\abs{\nabla \bfu} \chi_{\set{\bfxi \not=
        \bfxi_\lambda}} ) \,dx + \frac{1}{2} \dashint_B
    \psi^*(\abs{\nabla \bfxi})\,dx =: II_1 + II_2,
  \end{align*}
  where $c$ depends on $\abs{\mathcal{A}}$, $\Delta_2(\psi,\psi^*)$.
  With H{\"o}lder's inequality we get
  \begin{align*}
    II_1 &\leq c \bigg(\dashint_B \big( \psi(\abs{\nabla \bfu)})
    \big)^s\,dx\bigg)^{\frac 1s} \bigg( \frac{\abs{\set{\bfxi_\lambda
          \not= \bfxi}}}{\abs{B}} \bigg)^{1 - \frac 1s}.
  \end{align*}
  If follows from \eqref{eq:xi2},~\eqref{eq:choicegamma} and $\lambda
  \geq \gamma$ that
  \begin{align*}
    \frac{\abs{\set{\bfxi_\lambda \not= \bfxi}}}{\abs{B}} &\leq
    \frac{c \psi^*(\gamma)}{m_0 \psi^*(\lambda)} \leq \frac{c}{m_0}.
  \end{align*}
  Thus
  \begin{align*}
    II_1 &\leq c \bigg(\dashint_B \big( \psi(\abs{\nabla \bfu)})
    \big)^s\,dx\bigg)^{\frac 1s} \bigg( \frac{c}{m_0} \bigg)^{1 -
      \frac 1s}.
  \end{align*}
  We choose $m_0$ so large such that
  \begin{align*}
    II_1 &\leq \frac{\epsilon}{2} \bigg(\dashint_B \big(
    \psi(\abs{\nabla \bfu)}) \big)^s\,dx\bigg)^{\frac 1s}.
  \end{align*}
  Since $\bfu$ is almost $\mathcal{A}$-harmonic and
  $\norm{\nabla\bfxi_\lambda}_\infty \leq c\, \lambda \leq c\, 2^{m_0}
  \gamma$ we have
  \begin{align*}
    |I| &\leq \delta\, \dashint_{\widetilde B} \abs{\nabla \bfu}\,dx\,
    \norm{\nabla \bfxi_\lambda}_\infty \leq \delta\,
    \dashint_{\widetilde B} \abs{\nabla \bfu}\,dx\, c\,2^{m_0} \gamma.
  \end{align*}
  We apply Young's inequality and~\eqref{eq:choicegamma} to get
  \begin{align*}
    |I| &\leq \delta 2^{m_0} c\bigg(\dashint_{\widetilde B} \psi(
    \abs{\nabla \bfu})\,dx + \psi^*(\gamma) \bigg)
    \\
    &\leq \delta 2^{m_0} c \,\dashint_{\widetilde B} \psi( \abs{\nabla
      \bfu}) \,dx + \delta 2^{m_0}c \dashint_B \psi^*(\abs{\nabla
      \bfxi})\,dx.
  \end{align*}
  Now, we choose $\delta>0$ so small such that $\delta 2^{m_0} c\leq
  \epsilon/2$. Thus
  \begin{align*}
    |I| &\leq \frac{\epsilon}{2} \dashint_{\widetilde B}
    \psi( \abs{\nabla \bfu}) \,dx  +
    \frac{1}{2} \dashint_B \psi^*(\abs{\nabla \bfxi})\,dx.
  \end{align*}
  Combining the estimates for $I$, $II$ and $II_1$ we get
  \begin{align*}
    \dashint_B \!\mathcal{A}\nabla \bfu \!\cdot\! \nabla \bfxi\,dx
    \leq \epsilon \Bigg( \bigg(\dashint_{B} \big(\psi( \abs{\nabla
      \bfu}) \big)^s\,dx\bigg)^{\frac 1s} + \dashint_{\widetilde B}
    \psi( \abs{\nabla \bfu}) \,dx \Bigg) + \dashint_B
    \psi^*(\abs{\nabla \bfxi})\,dx.
  \end{align*}
  Now taking the supremum over all $\bfxi \in
  C^\infty_0(B)$ and using~\eqref{eq:Liptrunvi} we get
  \begin{align*}
    \dashint_B \psi(\abs{\nabla \bfw})\,dx \leq \epsilon \Bigg(
    \bigg(\dashint_{B} \big(\psi( \abs{\nabla \bfu})
    \big)^s\,dx\bigg)^{\frac 1s} + \dashint_{\widetilde B} \psi(
    \abs{\nabla \bfu}) \,dx \Bigg).
  \end{align*}
  The claim follows by Poincar{\'e}  inequality, see
Theorem~\ref{thm:poincare}.
\end{proof}

\section{\texorpdfstring{Almost $\mathcal{A}$-harmonicity}{Almost A-harmonicity}}
\label{sec:almost}

The following result is a special case
of~\cite[Lemma~A.2]{DieKapSch11}.
\begin{lemma}
  \label{lem:manyavg}
  Let $B \subset \Rn$ be a ball and $\bfw \in W^{1,\phi}(B)$. Then 
  \begin{align*}
    \dashint_B \abs{\bfV(\nabla \bfw) - \mean{\bfV(\nabla
        \bfw)}_B}^2\,dx &\sim \dashint_B \abs{\bfV(\nabla \bfw) -
      \bfV(\mean{\nabla \bfw}_B)}^2\,dx.
  \end{align*}
  The constants are independent of $B$ and $\bfw$; they only depend on
  the characteristics of~$\phi$.
\end{lemma}
\begin{lemma}
  \label{lem:avg}
  There exists $\delta>0$, which only depends on the characteristics
  of~$\phi$, such that for every ball~$B$ with $B \subset
  \subset\Omega$ and every $\bfu\in W^{1,\phi}(B)$ the estimate
  \begin{align}
    \label{eq:avg}
    \dashint_B \abs{\bfV(\nabla \bfu) - \mean{\bfV(\nabla
        \bfu)}_B}^2\,dx &\leq \delta\, \dashint_B \abs{\bfV(\nabla
      \bfu)}^2\,dx
  \end{align}
  implies
  \begin{gather}
    \label{eq:avg2}
    \dashint_B \abs{\bfV(\nabla \bfu)}^2\,dx \leq 4\,
    \abs{\bfV(\mean{\nabla \bfu}_B)}^2,
    \\
    \label{eq:avg3}
    \dashint_B \abs{\bfV(\nabla \bfu) - \mean{\bfV(\nabla
        \bfu)}_B}^2\,dx \leq 4\,\delta\, \abs{\bfV(\mean{\nabla
        \bfu}_B)}^2.
  \end{gather}
\end{lemma}
\begin{proof}
  It follows from~\eqref{eq:avg} and Lemma~\ref{lem:manyavg} that
  \begin{align*}
    \dashint_B \abs{\bfV(\nabla \bfu)}^2\,dx &\leq 2\,\dashint_B
    \abs{\bfV(\nabla \bfu) - \bfV(\mean{\nabla \bfu}_B)}^2\,dx +
    2\,\abs{\bfV(\mean{\nabla \bfu}_B)}^2
    \\
    &\leq c\,\dashint_B \abs{\bfV(\nabla \bfu) - \mean{\bfV(\nabla
        \bfu)}_B}^2\,dx +  2\,\abs{\bfV(\mean{\nabla \bfu}_B)}^2
    \\
    &\leq \delta\,c\, \dashint_B \abs{\bfV(\nabla \bfu)}^2\,dx +
    2\,\abs{\bfV(\mean{\nabla \bfu}_B)}^2.
  \end{align*}
  For small~$\delta$ we absorb the first term of the right-hand side to
  get~\eqref{eq:avg2}. The remaining estimate~\eqref{eq:avg3} is a
  combination of~\eqref{eq:avg} and~\eqref{eq:avg2}.
\end{proof}
\begin{lemma}
\label{lem:ab}
 Let $\bfu$ be a local minimizer of $\mathcal{F}$. Then
  for every ball~$B$ with $2B \subset\subset \Omega$ and every $\bfQ \in
  \RNn$ it  holds
  \begin{align*}
    \dashint_B \phi_{\abs{\bfQ}}(\abs{\nabla \bfu - \bfQ})\,dx &\leq c\,
    \phi_{\abs{\bfQ}} \bigg( \dashint_{2B} \abs{\nabla \bfu - \bfQ} \,dx
    \bigg).
  \end{align*}
\end{lemma}
\begin{proof}
  From Corollary~\ref{cor:gehring2} we get 
\begin{align*}
  \dashint_B \phi_{\abs{\bfQ}}(\abs{\nabla \bfu - \bfQ})\,dx &\leq c\,
  \bigg(\dashint_{2B} \phi_{\abs{\bfQ}}(\abs{\nabla \bfu -
    \bfQ})^{\alpha}\,dx\bigg)^{\frac{1}{\alpha}} .
\end{align*}
We can apply then Corollary~3.4 in~\cite{DieKapSch11} to
conclude.
\end{proof}
\begin{lemma}
  \label{lem:lowavg}
  For all $\epsilon>0$ there exists $\delta>0$, which only depends
  on~$\epsilon$ and the characteristics of~$\phi$, such that for every
  ball~$B$ with $B \subset \subset \Omega$ and every $\bfu\in W^{1,\phi}(B)$
  \begin{align}
    \label{eq:lowavg}
    \dashint_B \abs{\bfV(\nabla \bfu) - \mean{\bfV(\nabla
        \bfu)}_B}^2\,dx &\leq \delta\, \dashint_B \abs{\bfV(\nabla
      \bfu)}^2\,dx
  \end{align}
  implies
  \begin{align}
    \label{eq:lowavg2}
    \dashint_B \abs{\nabla \bfu - \mean{\nabla \bfu}_B}\,dx \leq
    \epsilon\, \abs{\mean{\nabla \bfu}_B}.
  \end{align}
\end{lemma}
\begin{proof}
  Let $\bfQ = \mean{\nabla \bfu}_B$. Then, by Jensen's
  inequality and Lemma~\ref{lem:avg} we get
  \begin{align*}
    \phi_{\abs{\mean{\nabla \bfu}_B}}\bigg( \dashint_B \abs{\nabla
      \bfu - \mean{\nabla \bfu}_B}\,dx \bigg) &\leq \dashint_B
    \phi_{\abs{\mean{\nabla \bfu}_B}}(\abs{\nabla \bfu - \mean{\nabla
        \bfu}_B})\,dx
    \\
    &\leq c\, \dashint_B \abs{\bfV(\nabla \bfu) - \bfV(\mean{\nabla
        \bfu}_B)}^2\,dx
    \\
    &\leq \delta\,c\, \abs{\bfV(\mean{\nabla \bfu}_B)}^2
    \\
    &\leq \delta\,c\, \phi(\abs{\mean{\nabla \bfu}_B})
    \\
    &\leq \delta\,c\, \phi_{\abs{\mean{\nabla
          \bfu}_B}}(\abs{\mean{\nabla \bfu}_B}).
  \end{align*}
  For the last inequality we used the fact that $\phi(a) \sim \phi_a(a)$ for $a \geq 0$. Using
  the $\Delta_2$-condition of~$\phi_{\abs{\mean{\nabla \bfu}_B}}$ it follows that for
every~$\epsilon>0$ there exists a ~$\delta>0$ such that
  \begin{align*}
    \dashint_B \abs{\nabla \bfu - \mean{\nabla \bfu}_B} \,dx &\leq
    \epsilon\,\abs{\mean{\nabla \bfu}_B}.
  \end{align*}
\end{proof}
Note that the smallness assumption in~\eqref{eq:lowavg} automatically
implies that $\mean{\nabla \bfu}_B\not=0$ (unless $\nabla \bfu=0$ on
$B$). So the smallness assumption ensures that in some sense in the
non-degenerate situation.
\begin{lemma}
  \label{lem:almost}
  For all $\epsilon>0$ there exists $\delta >0$ such that for every
  local minimizer $\bfu\in W^{1,\phi}_{loc}(\Omega)$ of $\mathcal{F}$
  and every ball~$B$ with $2B \subset\subset \Omega$ and
  \begin{align}
    \label{eq:osc_small2}
    \dashint_{2B} \abs{\bfV(\nabla \bfu) - \mean{\bfV(\nabla
        \bfu)}_{2B}}^2\,dx &\leq \delta\, \dashint_{2B}
    \abs{\bfV(\nabla \bfu)}^2\,dx
  \end{align}
  there holds
  \begin{align}
    \label{eq:almost}
    \biggabs{\dashint_B D^2 f(\bfQ)(\nabla \bfu - \bfQ, \nabla
      \bfxi\,dx} \leq \epsilon\, \phi''(\abs{\bfQ}) \dashint_{2B}
    \abs{\nabla \bfu - \bfQ}\,dx \norm{\nabla
      \bfxi}_\infty .
  \end{align}
  for every $\bfxi \in C^\infty_0(B)$, where $\bfQ := \mean{\nabla
    \bfu}_{2B}$. In particular, $\bfu$ is almost $\calA$-harmonic (in
  the sense of Theorem~\ref{thm:Aappr_phi}), with $\calA=D^2
  f(\bfQ)/\phi''(\abs{\bfQ})$.
\end{lemma}
\begin{proof}
  Let $\epsilon>0$. Without loss of generality we can assume
  that~$\delta>0$ is so small that the Lemmas~\ref{lem:avg} and
  \ref{lem:lowavg} give
  \begin{align}
    \label{eq:sm}
    \dashint_{2B} \abs{\bfV(\nabla \bfu)}^2\,dx \leq 4\,
    \abs{\bfV(\bfQ)}^2,\\
    \label{eq:almost3}    
    \dashint_{2B} \abs{\nabla \bfu - \bfQ}\,dx \leq \epsilon\,
    \abs{\bfQ}.
  \end{align}
  From the last inequality we deduce
  \begin{align}
    \label{eq:equi}
    \phi''(\abs{\bfQ})\, \bigg(\dashint_{2B} \abs{\nabla \bfu -
      \bfQ}\,dx\bigg)^2 \sim \phi_{\abs{\bfQ}}\bigg(\dashint_{2B}
    \abs{\nabla \bfu - \bfQ}\,dx\bigg).
  \end{align}
  Since the estimate~\eqref{eq:almost} is homogeneous with respect to
  $\norm{\nabla \bfxi}_\infty$, it suffices to show
  that~\eqref{eq:almost} holds for all $\bfxi \in C^\infty_0(B)$ with
  $\norm{\nabla \bfxi}_\infty= \dashint_{2B} \abs{\nabla \bfu -
    \bfQ}\,dx$. Hence, because of \eqref{eq:equi} it suffices to prove
  \begin{align}
    \label{eq:almost2}
    \biggabs{\dashint_B D^2 f(\bfQ)(\nabla \bfu - \bfQ, \nabla
      \bfxi)\,dx} \leq \epsilon\,c\, \phi_{\abs{\bfQ}}\bigg(
    \dashint_{2B} \abs{\nabla \bfu - \bfQ}\,dx \bigg)
  \end{align}
  for all such $\bfxi$. We define
  \begin{align*}
    B^\geq &:= \bigset{ x\in B\,:\, \abs{\nabla \bfu - \bfQ} \geq \tfrac{1}{2}
      \abs{\bfQ}},
    \\
    B^< &:= \bigset{ x\in B\,:\, \abs{\nabla \bfu - \bfQ}
      < \tfrac{1}{2} \abs{\bfQ}}.
  \end{align*}
  From the Euler-Lagrange equation we get $\int_B \big( Df(\nabla
  \bfv) - Df(\bfQ)\big):\nabla \bfxi\,dx = 0$, and therefore
  \begin{align*}
    \dashint_B D^2 f(\bfQ)&(\nabla \bfu - \bfQ, \nabla\bfxi)\,dx
    \\
    &= \dashint_B \int_0^1 \big(D^2 f(\bfQ) - D^2 f(\bfQ +
    \theta(\nabla \bfu - \bfQ)) \big) (\nabla \bfu - \bfQ, \nabla
    \bfxi)\, d\,\theta \,dx.
  \end{align*}
  We split the right-hand side into the integral $I$ over $B^\geq$ and
  the integral $II$ over $B^<$. Using \ref{itm:D2fbnd} we get
  \begin{align*}
    \abs{I} &\leq c\, \dashint_B \chi_{B^\geq} \int_0^1 \big(
    \phi''(\abs{\bfQ}) + \phi''(\abs{\bfQ+\theta(\nabla \bfu - \bfQ)})
    \big)\,d\theta\ \abs{\nabla \bfu - \bfQ} \abs{\nabla \bfxi}\,dx
    \\
    &\leq c\, \dashint_B \chi_{B^\geq} \big(\phi''(\abs{\bfQ}) +
    \phi''(\abs{\bfQ} + \abs{\nabla \bfu - \bfQ}) \big) \abs{\nabla
      \bfu - \bfQ} \abs{\nabla \bfxi}\,dx
    \\
    &\leq c\, \dashint_B \chi_{B^\geq} \big(\abs{\nabla \bfu - \bfQ}
    \phi'(\abs{\bfQ}) + \phi'_{\abs{\bfQ}} (\abs{\nabla \bfu - \bfQ})
    \abs{\bfQ} \big)\,dx\ \frac{\norm{\nabla \bfxi}_\infty}{\abs{\bfQ}}
    \\
    &\leq \epsilon\,c\, \dashint_B \chi_{B^\geq} \big(\abs{\nabla \bfu
      - \bfQ} \phi'(\abs{\bfQ}) + \phi'_{\abs{\bfQ}} (\abs{\nabla \bfu
      - \bfQ}) \abs{\bfQ} \big)\,dx.
  \end{align*}
  We used Lemma \ref{lem:phi_l_prop} for the second, Assumption
  \ref{ass:phi} for the third and \eqref{eq:almost3} for the last
  estimate. Now, using $\abs{\bfQ} \leq 2\, \abs{\nabla \bfu - \bfQ}$
  on $B^\geq$ and $\phi_a(t) \sim \phi(t)$ for $0 \le a \le t$ we get
  \begin{align*}
    \abs{I} &\leq \epsilon\,c\, \dashint_B \chi_{B^\geq}
    \big(\phi(\abs{\nabla \bfu - \bfQ}) + \phi_{\abs{\bfQ}}
    (\abs{\nabla \bfu - \bfQ}) \big)\,dx
    \\
    &\leq \epsilon\,c\, \dashint_B \phi_{\abs{\bfQ}} (\abs{\nabla
      \bfu - \bfQ}) \,dx.
  \end{align*}
  Let us estimate the modulus of $II$.  Using~\ref{itm:D2fdiff} and
  $\abs{\nabla \bfu - \bfQ} < \frac{1}{2} \abs{\bfQ}$ on $B^<$ we get
  \begin{align*}
    \abs{II} &\leq c\, \dashint_B \chi_{B^<} \phi''(\abs{\bfQ})
    \abs{\bfQ}^{-\beta_1} \abs{\nabla \bfu - \bfQ}^{1+\beta_1}
    \abs{\nabla \bfxi}\,dx,
  \end{align*}
  where $\beta_1 := \min \set{s_0, \beta}$ with the constant $s_0$
  from Corollary~\ref{cor:gehring3}. Using Young's inequality we get
  \begin{align*}
    \abs{II} &\leq \gamma \phi''(\abs{\bfQ}) \norm{\nabla
      \bfxi}_\infty^2 + c_\gamma \dashint_B \chi_{B^<}
    \phi''(\abs{\bfQ}) \abs{\bfQ}^{-2\beta_1} \abs{\nabla \bfu -
      \bfQ}^{2(1+\beta_1)}\,dx
    \\
    &\leq \gamma\, c\, \phi_{\abs{\bfQ}}(\norm{\nabla \bfxi}_\infty) +
    c_\gamma (\phi(\abs{\bfQ}))^{-\beta_1} \dashint_B \chi_{B^<}
    \big(\phi''(\abs{\bfQ}) \abs{\nabla \bfu - \bfQ}^2
    \big)^{1+\beta_1}\,dx
    \\
    &\leq \gamma\, c\, \dashint_{2B}\phi_{\abs{\bfQ}}(\abs{\nabla \bfu -
      \bfQ})\,dx + c_\gamma (\phi(\abs{\bfQ}))^{-\beta_1} \dashint_B
    \chi_{B^<} \big( \phi_{\abs{\bfQ}}(\abs{\nabla \bfu - \bfQ})
    \big)^{1+\beta_1}\,dx
    \\
    &\leq \gamma\,c\, \dashint_{2B} \abs{\bfV(\nabla \bfu) -
      \bfV(\bfQ)}^2\,dx + c_\gamma (\phi(\abs{\bfQ}))^{-\beta_1}
    \dashint_B \abs{\bfV(\nabla \bfu) -
      \bfV(\bfQ)}^{2(1+\beta_1)}\,dx.
  \end{align*}
  Here we used \eqref{eq:equi} for the second and Jensen's inequality,
  $\phi''(a) t^2 \sim \phi_a(t)$ for $0 \leq t \leq a$ and
  $\abs{\nabla \bfu - \bfQ} < \frac{1}{2} \abs{\bfQ}$ on $B^<$ for the
  third estimate. With the help of Corollary~\ref{cor:gehring3} we get
  \begin{align*}
    \abs{II} &\leq \gamma\,c\, \dashint_{2B} \abs{\bfV(\nabla \bfu) -
      \bfV(\bfQ)}^2\,dx + c_\gamma (\phi(\abs{\bfQ}))^{-\beta_1}
    \bigg( \dashint_{2B} \abs{\bfV(\nabla \bfu) - \bfV(\bfQ)}^2 \,dx
    \bigg)^{1+\beta_1}.
  \end{align*}
  Using the assumption~\eqref{eq:osc_small2}, Lemma~\ref{lem:manyavg}
  and \eqref{eq:sm} it follows that
  \begin{align*}
    \abs{II} &\leq \gamma\,c\, \dashint_{2B} \abs{\bfV(\nabla \bfu) -
      \bfV(\bfQ)}^2\,dx + c_\gamma  \delta^{\beta_1}
    \dashint_{2B} \abs{\bfV(\nabla \bfu) -
      \bfV(\bfQ)}^2\,dx.
  \end{align*}
Choosing $\gamma>0$ and then $\delta>0$ small enough we get the assertion.
\end{proof}

\section{Excess decay estimate}
\label{sec:comparison}

In this section we will focus on the excess decay estimate.
Therefore, we compare the almost harmonic solution with its harmonic
approximation.
\begin{proposition}
  \label{pro:tau}
  For all $\epsilon>0$, there exists
  $\delta=\delta(\phi,\epsilon)>0$ such that the following is true: if
  for some ball~$B$ with $2B \subset \subset \Omega$ the smallness
  assumption~\eqref{eq:osc_small2} holds true, then for every $\tau
  \in (0,1]$
  \begin{align}
    \label{eq:tau1}
    \Phi(\tau B, \bfu)\le c\, \tau^2 (1+ \epsilon\,
    \tau^{-n-2}\big) \, \Phi(2B, \bfu),
  \end{align}
  where $c$ depends only on the characteristics of~$\phi$ and is
  independent of $\epsilon$.
\end{proposition}
\begin{proof}
  It suffices to consider the case $\tau \leq \frac{1}{2}$.  Let $s_0$
  be as in Corollary~\ref{cor:gehring3}. Let $\bfq$ be a linear
  function such that $\mean{\bfu - \bfq}_{2B} = 0$ and $\bfQ := \nabla
  \bfq = \mean{\nabla \bfu}_{2B}$. Define $\bfz := \bfu - \bfq$. Let
  $\bfh$ be the harmonic approximation of $\bfz$ with $\bfh=\bfz$ on
  $\partial B$. It follows from Lemma~\ref{lem:almost} that $\bfz$ is
  almost $\mathcal{A}$-harmonic with $\calA = D^2f(\bfQ)/
  \phi''(\abs{\bfQ})$.  Thus by Theorem \ref{thm:Aappr_phi} for
  suitable $\delta=\delta(\phi,\epsilon)$ and by
  Theorem~\ref{thm:Aappr_phi} the $\calA$-harmonic approximation
  $\bfh$ satisfies
  \begin{align*}
    \dashint_B \phi_{\abs{\bfQ}}(\abs{\nabla \bfz - \nabla \bfh}) \,dx
    &\leq \epsilon \Bigg( \bigg( \dashint_{B}
    \phi_{\abs{\bfQ}}^{s_0}(\abs{\nabla \bfu -
      \bfQ})\,dx\bigg)^{\frac{1}{s_0}} + \dashint_{2B}
    \phi_{\abs{\bfQ}}(\abs{\nabla \bfu - \bfQ})\,dx \Bigg).
  \end{align*}
  Now, it follows by Corollary~\ref{cor:gehring3} that
  \begin{align}
    \label{eq:2}
    \dashint_B \phi_{\abs{\bfQ}}(\abs{\nabla \bfz - \nabla \bfh}) \,dx
    &\leq c\, \epsilon\, \Phi(2B, \bfu).
  \end{align}
  Since $\nabla \bfz = \nabla \bfu -\bfQ$ and $\mean{\nabla \bfz}_{\tau
    B} = \mean{\nabla \bfu}_{\tau B} - \bfQ$, we get
  \begin{align*}
    \Phi(\tau B, \bfu) &\leq c\, \dashint_{\tau B}
    \phi_{\abs{\bfQ}}( \abs{\nabla \bfz - \mean{\nabla
        \bfz}_{\tau B}})\,dx 
    \\
    &\leq c\,  \dashint_{\tau B} \phi_{\abs{\bfQ}}( \abs{\nabla
      \bfh - \mean{\nabla \bfh}_{\tau B}})\,dx  +
    c\,  \dashint_{\tau B} \phi_{\abs{\bfQ}}( \abs{\nabla \bfz
      - \nabla \bfh})\,dx 
    \\
    &=: I + II.
  \end{align*}
  For the second estimate we used Jensen's inequality.
  Using~\eqref{eq:2} we obtain
  \begin{align*}
    II &\leq \tau^{-n} c\,  \dashint_{B}
    \phi_{\abs{\bfQ}}( \abs{\nabla \bfz - \nabla \bfh})\,dx
    \leq \tau^{-n} c\,
    \epsilon\, \Phi(2B, \bfu).
  \end{align*}
  By the interior regularity of the $\calA$-harmonic function $\bfh$, \cite{Gia82},
  and $\tau \leq \frac 12$ it holds that
 
  \begin{align*}
    \sup_{\tau B} \abs{\nabla \bfh - \mean{\nabla \bfh}_{\tau B}} &\leq
    c\, \tau \dashint_{B} \abs{\nabla \bfh - \mean{\nabla
        \bfh}_{B}}\,dx.
  \end{align*}
  This proves
  \begin{align*}
    I &\leq c\, \phi_{\abs{\bfQ}}\bigg( \tau \dashint_{B} \abs{\nabla
      \bfh - \mean{\nabla \bfh}_{B}}\,dx \bigg).
  \end{align*}
  Using the estimate $\psi(s t) \leq s \psi(t)$ for any $s \in [0,1]$,
  $t \geq 0$ and any N-function $\psi$, we would get a factor~$\tau$
  in the estimate of $I$. However, to produce a factor~$\tau^2$, we
  have to work differently and use the improved estimate $\phi_a(s\,t)
  \leq c\, s^2 \phi_a(t)$ for all $s \in [0,1]$, $a \geq 0$ and $t \in
  [0,a]$. We begin with
  \begin{align*}
    \dashint_{B} \abs{\nabla \bfh - \mean{\nabla \bfh}_{B}}\,dx
    &\leq \dashint_{B} \abs{\nabla \bfz - \mean{\nabla
        \bfz}_{B}}\,dx + 2\,\dashint_{B} \abs{\nabla \bfz - \nabla
      \bfh}\,dx
    \\
    &= \dashint_{B} \abs{\nabla \bfu - \mean{\nabla
        \bfu}_{B}}\,dx + 2\,\dashint_{B} \abs{\nabla \bfz - \nabla
      \bfh}\,dx,
  \end{align*}
  which implies
  \begin{align*}
    I &\leq c\, \phi_{\abs{\bfQ}}\bigg( \tau \dashint_{B} \abs{\nabla
      \bfu - \mean{\nabla \bfu}_{B}}\,dx \bigg) + c\, \tau
    \phi_{\abs{\bfQ}}\bigg( \dashint_{B} \abs{\nabla \bfz -
      \nabla \bfh}\,dx \bigg).
  \end{align*}
  Due to~\eqref{eq:almost3}, we can use for the first term the
  improved estimate $\phi_a(s\,t) \leq c\, s^2 \phi_a(t)$, which gives
  \begin{align*}
    I &\leq c\, \tau^2\, \phi_{\abs{\bfQ}}\bigg( \dashint_{B}
    \abs{\nabla \bfu - \mean{\nabla \bfu}_{B}}\,dx \bigg) + c\,
    \tau\, \phi_{\abs{\bfQ}}\bigg( \dashint_{B} \abs{\nabla \bfz -
      \nabla \bfh}\,dx \bigg)
    \\
     &\leq c\, \tau^2\, \dashint_{B} \phi_{\abs{\bfQ}}\big(
    \abs{\nabla \bfu - \mean{\nabla \bfu}_{B}} \big) \,dx + c\,
    \tau\, \dashint_{B} \phi_{\abs{\bfQ}} \big(\abs{\nabla \bfz -
      \nabla \bfh} \big)\,dx.
  \end{align*}
  Thus using~\eqref{eq:2} we get
  \begin{align*}
    I &\leq c\, \tau^2\, \Phi(B,\bfu)+ c\, \tau\, \epsilon\,
    \Phi(2B, \bfu) \leq c\, \big(\tau^2\, + \epsilon\,
    \tau\big)\, \Phi(2B,\bfu).
  \end{align*}
  Combining the estimates for $I$ and
  $II$ we get the claim.
\end{proof}
It follows now,  by a series of standard arguments,  that for any $\beta
\in(0,1)$, there exists a suitable small  $\delta$ that
ensures local $C^{0,\beta}$-regularity of $\bfV(\nabla u)$, which
implies H{\"o}lder continuity of the gradients as well.
\begin{proposition}[Decay estimate]
  \label{pro:decay}
  For $ 0 <\beta < 1$ there exists $\delta = \delta(\phi, \beta) > 0$
  such that the following is true. If for some ball~$B \subset \Omega$
  the smallness assumption~\eqref{eq:osc_small2} holds true, then
 \begin{align}
   \label{eq:excessdecay1}
   \Phi(\rho B, \bfu)\le c\, \rho^{2\beta} \Phi(2B, \bfu)
 \end{align}
 for any $\rho \in (0,1]$, where $c=c(\phi)$ depends only on the
 characteristics of~$\phi$.
\end{proposition}
\begin{proof}
  Due to our assumption, we can apply Proposition~\ref{pro:tau} for
  any~$\tau$. Let $\gamma(\epsilon,\tau):=c\, \tau^2\,(1+\epsilon\,
  \tau^{-n-2})$ as in~\eqref{eq:tau1}. Let us fix
  $\tau>0$ and $\epsilon>0$, such that $\gamma(\epsilon,\tau) \leq
  \min \set{(\tau/2)^{2\beta}, \frac 14}$. Let
  $\delta=\delta(\phi,\epsilon)$ chosen accordingly to
  Proposition~\ref{pro:tau} and also
  so small that $(1+\tau^{-n/2})\delta^{1/2} \leq
  \frac{1}{2}$. By Proposition~\ref{pro:tau} we have
  \begin{align}
    \label{eq:tau2decay}
    \Phi(\tau B, \bfu) &\leq \min
    \set{(\tau/2)^{2\beta}, \tfrac 14}\, \Phi(2B, \bfu).
  \end{align}
  We claim that the smallness assumption is inherited from~$2B$ to
  $\tau B$, so that we can iterate~\eqref{eq:tau2decay}. For this we
  estimate with the help of our smallness assumption
  \begin{align*}
    \bigg(\dashint_{2B} \abs{\bfV(\nabla \bfu)}^2\,dx \bigg)^{\frac
      12} &\leq \big(\Phi(2B, \bfu)\big)^{\frac 12} +
    \abs{\mean{\bfV(\nabla \bfv)}_{2B} - \mean{\bfV(\nabla
        \bfv)}_{\tau B}} + \bigg(\dashint_{\tau B}
    \abs{\bfV(\nabla \bfv)}^2\,dx \bigg)^{\frac 12}
    \\
    &\leq \big(\Phi(2B, \bfu)\big)^{\frac 12} + \tau^{-n/2}
    \big(\Phi(2B, \bfu)\big)^{\frac 12} + \bigg(\dashint_{\tau B}
    \abs{\bfV(\nabla \bfv)}^2\,dx \bigg)^{\frac 12}
    \\
    &\leq \big(1 + \tau^{-n/2}\big) \delta^{1/2}
    \bigg(\dashint_{2B} \abs{\bfV(\nabla \bfu)}^2\,dx\bigg)^{\frac 12}
    + \bigg(\dashint_{\tau B} \abs{\bfV(\nabla \bfv)}^2\,dx
    \bigg)^{\frac 12}.
  \end{align*}
  Using $(1+\tau^{-n/2})\delta^{1/2} \leq \frac{1}{2}$, we get
  \begin{align*}
    \dashint_{2B} \abs{\bfV(\nabla \bfu)}^2\,dx &\leq 4\,
    \dashint_{\tau B} \abs{\bfV(\nabla \bfv)}^2\,dx.
  \end{align*}
  Now~\eqref{eq:tau2decay} and the previous estimate imply
  \begin{align*}
    \Phi(\tau B, \bfu)  &\leq \frac 14
    \Phi(2B, \bfu) \leq \frac 14 \delta \dashint_{2B}
    \abs{\bfV(\nabla \bfu)}^2\,dx \leq \delta \dashint_{\tau B}
    \abs{\bfV(\nabla \bfu)}^2\,dx.
  \end{align*}
  In particular, the smallness assumption is also satisfied for $\tau
  B$. So by induction we get
  \begin{align}
    \Phi((\tau/2)^k 2B, \bfu) &\leq \min \set{(\tau/2)^{2\beta
        k}, 4^{-k}}\, \Phi(2B,\bfu),
  \end{align}
  which is the desired claim.
\end{proof}

Having the decay estimate, it is easy to proove our Main Theorem.
\begin{proof}[Proof of the Main Theorem~\ref{thm:main}] 
  We can assume that \eqref{eq:osc_small2} is satisfied with a strict inequality. By
  continuity, \eqref{eq:osc_small2},  holds for $B=B(x)$ and all
  $x$ in some neighborhood of $x_0$. By Proposition \ref{pro:decay} and
  Campanato's characterisation of H{\"o}lder continuity we
  deduce that $\bfV(\nabla \bfu)$ is $\beta$-H{\"o}lder continuous in a
  neighbourhood of $x_0$. 
\end{proof}

\bibliographystyle{plain}
\bibliography{lars}

\end{document}